\documentclass[11pt,a4paper]{article}

\usepackage[margin=27mm]{geometry}
\usepackage{amsmath,amssymb,amsthm,mathtools,mathrsfs}
\usepackage{microtype}
\usepackage{needspace}
\usepackage[hidelinks,unicode]{hyperref}
\usepackage{aliascnt}
\usepackage[nameinlink,noabbrev]{cleveref}
\usepackage{indentfirst}
\usepackage{enumitem}

\hypersetup{
  pdftitle={Almost Homogeneity Characterizes Almost Bi-Lipschitz Euclidean Embeddability},
  pdfauthor={Ryo Fukano}
}

\allowdisplaybreaks[2]

\theoremstyle{plain}
\newtheorem{theorem}{Theorem}[section]

\newaliascnt{lemma}{theorem}
\newtheorem{lemma}[lemma]{Lemma}
\aliascntresetthe{lemma}

\newaliascnt{proposition}{theorem}
\newtheorem{proposition}[proposition]{Proposition}
\aliascntresetthe{proposition}

\newaliascnt{corollary}{theorem}
\newtheorem{corollary}[corollary]{Corollary}
\aliascntresetthe{corollary}

\theoremstyle{definition}
\newaliascnt{definition}{theorem}
\newtheorem{definition}[definition]{Definition}
\aliascntresetthe{definition}

\theoremstyle{remark}
\newaliascnt{remark}{theorem}
\newtheorem{remark}[remark]{Remark}
\aliascntresetthe{remark}
\crefname{theorem}{theorem}{theorems}
\crefname{lemma}{lemma}{lemmas}
\crefname{proposition}{proposition}{propositions}
\crefname{corollary}{corollary}{corollaries}
\crefname{definition}{definition}{definitions}
\crefname{remark}{remark}{remarks}
\crefname{section}{section}{sections}
\crefname{subsection}{subsection}{subsections}
\crefname{equation}{equation}{equations}
\Crefname{theorem}{Theorem}{Theorems}
\Crefname{lemma}{Lemma}{Lemmas}
\Crefname{proposition}{Proposition}{Propositions}
\Crefname{corollary}{Corollary}{Corollaries}
\Crefname{definition}{Definition}{Definitions}
\Crefname{remark}{Remark}{Remarks}
\Crefname{section}{Section}{Sections}
\Crefname{subsection}{Subsection}{Subsections}
\Crefname{equation}{Equation}{Equations}

\newcommand{\R}{\mathbb R}
\newcommand{\N}{\mathbb N}
\newcommand{\Z}{\mathbb Z}
\newcommand{\norm}[1]{\lVert#1\rVert}
\newcommand{\card}[1]{\lvert#1\rvert}
\newcommand{\Lip}{\operatorname{Lip}}
\newcommand{\diam}{\operatorname{diam}}
\newcommand{\cover}{\mathcal N}
\newcommand{\Ld}{\Lambda}

\title{Almost Homogeneity Characterizes Almost Bi-Lipschitz Euclidean Embeddability}
\author{Ryo Fukano}
\date{October 30, 2026}

\begin{document}
\maketitle
\begin{NoHyper}
\begingroup
\renewcommand{\thefootnote}{}
\endgroup
\end{NoHyper}

\begin{abstract}
A metric space admits an almost bi-Lipschitz embedding into a Euclidean space if and only if it is almost homogeneous.
 We prove the sufficient direction, complementing the known necessary condition. More quantitatively, we bound the target dimension and the embedding constants for every logarithmic exponent $\gamma>1$. 
The bounds depend only on the covering parameters in the almost homogeneity condition. 
In particular every doubling metric space admits such an embedding.
This answers a question of Margaris and Robinson and the almost bi-Lipschitz part of a question of Olson and Robinson. 
As a consequence, the difference set of the image is doubling, which answers the difference-set questions of Olson and Robinson and of Margaris.
\end{abstract}

\section{Introduction}\label{sec:intro}

The main result of this paper is to show that almost bi-Lipschitz embeddability into Euclidean spaces is equivalent to almost homogeneity.
An almost bi-Lipschitz embedding is characterized by an upper Lipschitz bound and a lower bound that permits a logarithmic loss.

Assouad's theorem
\cite{Assouad1983} gives a Euclidean bi-Lipschitz embedding after
snowflaking. More precisely let $(X,d)$ be doubling and let
$0<\theta<1$. Then $(X,d^\theta)$ admits a bi-Lipschitz embedding into a
Euclidean space. However the doubling condition alone does not imply
the existence of a bi-Lipschitz embedding for the original metric $d$
\cite{Laakso2002,Semmes1996}.

Olson and Robinson \cite[Theorem~2.3]{OlsonRobinson2010} constructed
almost bi-Lipschitz embeddings of almost homogeneous metric spaces into
Hilbert space. For linear embeddings into Euclidean space, they assumed
almost homogeneity of the difference set $A-A:=\{x-y:x,y\in A\}$ of a
compact subset $A$ of a Hilbert space
\cite[Theorem~5.6]{OlsonRobinson2010}. Robinson \cite{Robinson2014} and
Margaris \cite{Margaris2021} subsequently improved the logarithmic
exponents and covering conditions in this direction. These results impose covering conditions on the difference set in
addition to conditions on the set itself.

We show that almost homogeneity of the metric space itself is sufficient
for an almost bi-Lipschitz embedding into Euclidean space. For every
$\gamma>1$ we bound the target dimension in terms of the covering
parameters. This answers \cite[Section~3, question~(2)]{MargarisRobinson2020}
and the almost bi-Lipschitz part of
\cite[Section~9, Question~(2)]{OlsonRobinson2010}.

We first prove the result for finite subspaces by a scale-by-scale construction using Gumbel variables and the Lov\'asz local lemma. 
Uniform estimates then yield the result for general metric spaces.

We now state the main results. The relevant definitions appear in
\cref{def:ah,def:doubling,def:abl}.

\Needspace{7\baselineskip}
\begin{theorem}\label{thm:char}
For a metric space $X$, the following conditions are equivalent.
\begin{enumerate}[label=\textup{(\roman*)},leftmargin=2.2em]
\item $X$ is almost homogeneous.
\item There exist a positive integer $N$ and a finite $\gamma\geq0$ such
that $X$ admits a $\gamma$-almost bi-Lipschitz embedding into $\R^N$.
\end{enumerate}
\end{theorem}

The implication \textup{(ii)} $\Rightarrow$ \textup{(i)} is known. See
\cite[Section~3]{OlsonRobinson2010}. We prove the converse implication
\textup{(i)} $\Rightarrow$ \textup{(ii)} by the next theorem.

\Needspace{9\baselineskip}
\begin{theorem}\label{thm:quant}
Suppose that $X$ is $(M,s,\alpha,\beta)$-almost homogeneous. Put
$h:=\alpha+\beta$. For every $\gamma$ and every positive integer $N$
satisfying
\begin{equation}\label{eq:quantcond}
 \gamma>1,\qquad N>3s+\frac{4(s+h)}{\gamma-1},
\end{equation}
$X$ admits a $\gamma$-almost bi-Lipschitz embedding into $\R^N$.
The embedding constants can be chosen to depend only on
$M,s,\alpha,\beta,\gamma,N$.
\end{theorem}

\Needspace{5\baselineskip}
\begin{corollary}\label{cor:doubling}
Let $K\geq2$ and $\gamma>1$. Every $K$-doubling metric space admits a
$\gamma$-almost bi-Lipschitz embedding into $\R^N$ for every positive
integer $N$ satisfying $N>(3+4/(\gamma-1))\log_2K$.
The embedding constants depend only on $K,\gamma,N$.
\end{corollary}

The main theorem also implies that the difference set of the image is doubling.
This gives affirmative answers to
\cite[Section~9, Question~(3)]{OlsonRobinson2010} and
\cite[Question~4.2]{Margaris2020v1}. We state this consequence as
\cref{cor:diff}.

The paper is organized as follows. \Cref{sec:prelims} contains the
definitions and preliminary results. \Cref{sec:proof} gives the
construction on finite subspaces and the passage to general spaces.

\section{Preliminaries}\label{sec:prelims}

All logarithms are natural unless their base is specified.
We write $\N:=\{1,2,\ldots\}$ and let $\Z$ denote the integers.
Subsets of metric spaces are equipped with the restricted metric.
All normed spaces in this paper are over the real numbers.

\Needspace{8\baselineskip}
\begin{definition}\label{def:metric-notation}
Let $(X,d)$ be a metric space. For $x\in X$ and $R>0$ write
$B_X(x,R):=\{y\in X:d(x,y)<R\}$. For $A\subseteq X$ and $\rho>0$
let $\cover_X(A,\rho)$ be the minimum number of open $\rho$-balls
with centers in $X$ needed to cover $A$. Set $\cover_X(A,\rho)=\infty$
when no finite cover exists. Allowing centers in a larger ambient metric
space gives the same notion of almost homogeneity. Indeed every covering
ball that meets $A$ can be recentered at a point of $A$ after doubling
its radius. Put $\diam A:=\sup\{d(x,y):x,y\in A\}$ and set
$\diam\varnothing:=0$. We omit the subscript in $B_X(x,R)$ when the
ambient space is clear. For $t>0$ put $\Ld(t):=1+|\log t|$.
\end{definition}

\Needspace{9\baselineskip}
\begin{definition}[{\cite[(1.2)]{OlsonRobinson2010},
 \cite[Definition~1.3]{Margaris2020v1}}]\label{def:ah}
Let $(X,d)$ be a metric space. Let $M\geq1$ and
$s,\alpha,\beta\geq0$ be finite real numbers. We say that $X$ is
\emph{$(M,s,\alpha,\beta)$-almost homogeneous} if
\begin{equation}\label{eq:ah}
 \cover_X(B_X(x,R),\rho)
 \leq M\left(\frac R\rho\right)^s
 \Ld(R)^\alpha\Ld(\rho)^\beta
 \qquad(x\in X,\ 0<\rho<R).
\end{equation}
We call $X$ \emph{almost homogeneous} if such parameters exist. We call
it \emph{homogeneous} if they can be chosen with $\alpha=\beta=0$.
Our convention follows \cite[Definition~1.3]{Margaris2020v1}. The parameter $\alpha$
is attached to the outer radius and $\beta$ to the inner radius. The
order of these two parameters is reversed in
\cite[(1.2)]{OlsonRobinson2010}.
\end{definition}

\Needspace{6\baselineskip}
\begin{definition}[{cf.\ \cite[Section~1]{NaorNeiman2012}}]\label{def:doubling}
Let $K\geq1$. A metric space $(X,d)$ is \emph{$K$-doubling} if every
open ball in $X$ can be covered by at most $K$ open balls of half its
radius. The centers are taken in $X$. It is \emph{doubling} if it is
$K$-doubling for some finite $K$.
\end{definition}

\Needspace{6\baselineskip}
\begin{definition}\label{def:norms}
For $N\in\N$ and $v=(v_1,\ldots,v_N)\in\R^N$ put
$\norm v_2:=(\sum_{j=1}^N|v_j|^2)^{1/2}$ and
$\norm v_\infty:=\max_{1\leq j\leq N}|v_j|$.
For a map $f:(E,d)\to\R^N$ from a metric space let $\Lip_\infty(f)$
be the infimum of all $L\geq0$ such that
$\norm{f(x)-f(y)}_\infty\leq Ld(x,y)$ for all $x,y\in E$. Set
$\Lip_\infty(f)=\infty$ if no such finite $L$ exists.
\end{definition}

\Needspace{9\baselineskip}
\begin{definition}[{\cite[(1.1)]{OlsonRobinson2010}}]\label{def:abl}
For $\gamma\geq0$ a map $F:(X,d)\to\R^N$ is
\emph{$\gamma$-almost bi-Lipschitz} if there exist $c,C>0$ such that
\begin{equation}\label{eq:abl}
 c\,\frac{d(x,y)}{\Ld(d(x,y))^\gamma}
 \leq\norm{F(x)-F(y)}_2\leq C\,d(x,y)
 \qquad(x,y\in X,\ x\ne y).
\end{equation}
We call $F$ \emph{almost bi-Lipschitz} if this holds for some finite
$\gamma\geq0$. Maps into other normed spaces are defined similarly. The corresponding
norm replaces $\norm{\cdot}_2$.
When $\gamma=0$, this is the usual bi-Lipschitz condition.
\end{definition}

\begin{remark}[Logarithmic convention]\label{rem:conventions}
Earlier work uses the symmetric logarithm
\(\operatorname{slog}(t):=\log(t+t^{-1})\) in place of \(\Ld(t)\).
These two functions are comparable. Let \(a=|\log t|\). Then
\(\max\{a,\log2\}\leq \log(\mathrm e^a+\mathrm e^{-a})\leq a+\log2\). Hence
\(\frac{\log2}{1+\log2}\Ld(t)\leq\operatorname{slog}(t)\leq\Ld(t)\)
for every \(t>0\). Thus replacing \(\operatorname{slog}\) by \(\Ld\)
changes only the
constants in \cref{def:ah,def:abl}. In particular our definitions give
the same classes of almost homogeneous spaces and almost bi-Lipschitz
maps as those used in the open problems cited in the Introduction. The
choice of \(\Ld\) is only a normalization of the logarithmic factor.
\end{remark}

\Needspace{7\baselineskip}
\begin{definition}\label{def:nets}
Let $(X,d)$ be a metric space. Let $A\subseteq X$ and $\rho>0$.
A subset $S\subseteq A$ is \emph{$\rho$-separated} if every two distinct
points of $S$ are at distance at least $\rho$. It is a
\emph{$\rho$-net} in $A$ if
$A\subseteq\bigcup_{u\in S}B_X(u,\rho)$.
Maximality of a separated set is understood with respect to inclusion.
\end{definition}

\Needspace{10\baselineskip}
\begin{proposition}\label{prop:basic}
Let $X$ be a metric space.
\begin{enumerate}[label=\textup{(\roman*)},leftmargin=2.2em]
\item Let $K\geq1$. If $X$ is $K$-doubling then it is
$(K,\log_2K,0,0)$-almost homogeneous.
\item If $X$ is $(M,s,0,0)$-almost homogeneous then it is
$\lceil M2^s\rceil$-doubling. In particular, homogeneity and the
doubling condition are equivalent.
\item Every almost homogeneous metric space is separable.
\end{enumerate}
\end{proposition}
\begin{proof}
\textup{(i)} Let $0<\rho<R$ and put $j:=\lceil\log_2(R/\rho)\rceil$.
Iterating the doubling cover $j$ times covers $B(x,R)$ by at most
$K^j$ balls of radius $R2^{-j}\leq\rho$. Thus
$\cover_X(B(x,R),\rho)\leq K^j\leq K(R/\rho)^{\log_2K}$.

\textup{(ii)} Substitute $\rho=R/2$ into \eqref{eq:ah}.

\textup{(iii)} The case $X=\varnothing$ is immediate. Otherwise, fix
$o\in X$ and, for each $m,k\in\N$, cover $B(o,m)$ by finitely many
balls of radius $\min\{1/k,m/2\}$. Let $D$ be the union of all their
centers. Then $D$ is countable. Given $x\in X$ and $\varepsilon>0$,
choose $m>d(o,x)$ and $k>1/\varepsilon$. One of the centers $z\in D$
of the corresponding cover satisfies $d(x,z)<1/k<\varepsilon$.
Thus $D$ is dense.
\end{proof}

\Needspace{6\baselineskip}
\begin{lemma}\label{lem:nets}
Let $(X,d)$ be a metric space. Let $A\subseteq X$ and $\rho>0$.
Every maximal $\rho$-separated subset of $A$ is a $\rho$-net in $A$.
Moreover, if the cardinalities of all finite $\rho$-separated subsets
of $A$ have a common finite upper bound then $A$ has a finite maximal
$\rho$-separated subset.
\end{lemma}
\begin{proof}
If $S\subseteq A$ is $\rho$-separated and
$x\in A\setminus\bigcup_{u\in S}B_X(u,\rho)$ then $S\cup\{x\}$
is also $\rho$-separated. Thus a maximal $S$ is a $\rho$-net.
For the second assertion, start with the empty set and add one such
point $x$ whenever possible. The common upper bound on the cardinalities
of finite $\rho$-separated sets forces this procedure to stop after
finitely many steps, at which point the resulting set is maximal.
\end{proof}

\Needspace{8\baselineskip}
\begin{lemma}\label{lem:packing}
Suppose that $X$ is $(M,s,\alpha,\beta)$-almost homogeneous.
Let $E\subseteq X$ and $x\in E$. Let $0<\rho\leq R$.
Every $\rho$-separated set $S\subseteq B_E(x,R)$ is finite and satisfies
\[
 \card S\leq A\left(\frac R\rho\right)^s
 \Ld(R)^\alpha\Ld(\rho)^\beta,\qquad A:=M3^s(1+\log3)^\beta.
\]
\end{lemma}
\begin{proof}
Cover $B_X(x,R)$ by balls of radius $\rho/3$ with centers in $X$.
Two points in the same ball are at distance less than $2\rho/3$. Hence
each ball contains at most one point of $S$. Hence
$\card S\leq\cover_X(B_X(x,R),\rho/3)$.
The conclusion follows from \eqref{eq:ah} and
$\Ld(\rho/3)\leq(1+\log3)\Ld(\rho)$.
\end{proof}

\Needspace{12\baselineskip}
\begin{lemma}\label{lem:euclidean}
Let $N\in\N$.
\begin{enumerate}[label=\textup{(\roman*)},leftmargin=2.2em]
\item For $v\in\R^N$, we have
$\norm v_\infty\leq\norm v_2\leq\sqrt N\norm v_\infty$.
\item Let $u\in\R^N$ and $R,\rho>0$. If
$S\subseteq B_{\R^N}(u,R)$ is $\rho$-separated in the Euclidean metric then
$S$ is finite and
\begin{equation}\label{eq:euclideanpacking}
 \card S\leq(1+2R/\rho)^N.
\end{equation}
\item Every subset of $\R^N$ is $5^N$-doubling in the induced
Euclidean metric.
\end{enumerate}
\end{lemma}
\begin{proof}
\textup{(i)} For each $j$, we have
$|v_j|^2\leq\sum_{k=1}^N|v_k|^2$. We also have
$\sum_{j=1}^N|v_j|^2\leq N\max_j|v_j|^2$.
Taking square roots gives the claim.

\textup{(ii)} First suppose that $S=\{x_1,\ldots,x_m\}$ is finite.
Since $\norm{x_i-x_j}_2\geq\rho$ for $i\ne j$, the open balls
$B_{\R^N}(x_i,\rho/2)$ are pairwise disjoint. Suppose two such
balls contain a common point $z$. The triangle inequality would give
$\norm{x_i-x_j}_2\leq\norm{x_i-z}_2+\norm{z-x_j}_2<\rho$. This is
a contradiction. For each $x_i\in B_{\R^N}(u,R)$ we also have $B_{\R^N}(x_i,\rho/2)\subseteq B_{\R^N}(u,R+\rho/2)$.
Write $\operatorname{vol}_N$ for $N$-dimensional Lebesgue measure and
let $\omega_N>0$ be the measure of the open unit ball. Since the
measure of a ball is proportional to the $N$th power of its radius,
finite additivity and monotonicity give
\[
 m\omega_N(\rho/2)^N
 =\operatorname{vol}_N\!\left(\bigcup_{i=1}^mB_{\R^N}(x_i,\rho/2)\right)
 \leq\omega_N(R+\rho/2)^N.
\]
Dividing by $\omega_N(\rho/2)^N$ gives \eqref{eq:euclideanpacking}.
For general $S$, the same bound holds for every finite subset of $S$.
If $S$ were infinite we could choose a finite subset whose cardinality
exceeds the bound. The same contradiction follows if $S$ has more
points than the bound allows.

\textup{(iii)} Let $D\subseteq\R^N$ and $a\in D$. Let $R>0$.
By \textup{(ii)} every $R/2$-separated subset of $B_D(a,R)$ has
cardinality at most $(1+2R/(R/2))^N=5^N$.
By \cref{lem:nets} there is a finite maximal $R/2$-separated set. It is
an $R/2$-net. The open $R/2$-balls centered at its points cover
$B_D(a,R)$. This proves the claim.
\end{proof}

\Needspace{8\baselineskip}
\begin{lemma}\label{lem:gauge}
For every $\gamma\geq0$, there is a constant $\kappa_\gamma>0$ such that
\[
 \frac{t}{\Ld(t)^\gamma}
 \geq\kappa_\gamma\frac{r}{\Ld(r)^\gamma}
 \qquad(t\geq r>0).
\]
Moreover, every $\gamma$-almost bi-Lipschitz map from a metric space
into a normed space is a homeomorphism onto its image.
\end{lemma}
\begin{proof}
Write $t=rv$ with $v\geq1$. Since
$|\log(rv)|\leq|\log r|+\log v$, we have
$\Ld(rv)\leq(1+\log v)\Ld(r)$ and hence
\(\frac{t/\Ld(t)^\gamma}{r/\Ld(r)^\gamma}
\geq v/(1+\log v)^\gamma\).
The right-hand side is continuous and positive on $[1,\infty)$ and
tends to infinity as $v\to\infty$. Its infimum $\kappa_\gamma$
is therefore positive. This proves the asserted lower bound.

Now let $F:(X,d)\to V$ be a $\gamma$-almost bi-Lipschitz map with
constants $c,C>0$. By definition, $F$ is injective and continuous.
For every $\varepsilon>0$, the lower bound above shows that
$d(x,y)\geq\varepsilon$ implies
\(\norm{F(x)-F(y)}\geq
c\kappa_\gamma\varepsilon/\Ld(\varepsilon)^\gamma>0\).
The contrapositive gives continuity of the inverse map.
\end{proof}

\section{Euclidean embeddings}\label{sec:proof}

We first construct embeddings of finite subspaces $E$ of $X$ with
constants independent of $E$. At each scale, we use a Gumbel correction
map that depends only on independent variables indexed by nearby points
and apply the
Lov\'asz local lemma. For embedding constructions that similarly use
locality to control the number of dependent events, see also
\cite{AbrahamBartalNeiman2008} and
\cite[Section~1.2]{NaorNeiman2012}.
Unless otherwise stated, Lipschitz constants of $\R^N$-valued maps in
this section are taken with respect to the maximum norm.

\Needspace{16\baselineskip}
\subsection{Scales}
\begin{definition}\label{def:scales}
Fix a metric space $(X,d)$ satisfying \eqref{eq:ah}. Fix $\gamma,N$
satisfying \eqref{eq:quantcond}. Put $h:=\alpha+\beta$ and
$a:=(\gamma-1)/2$. Condition \eqref{eq:quantcond} is equivalent to
$\Delta:=aN-2h-s(2+3a)>0$. Define
$q:=h+\Delta/8$ and $\varepsilon:=\Delta/(8N)$.
Then $q>h$ and $\varepsilon>0$. Moreover
$\delta:=(a-2\varepsilon)N-2q-s(2+3a)=\Delta/2>0$.
In particular $a>2\varepsilon$. For $T\geq2$ and $n\in\Z$ put
$t_n:=T+|n|$. Define $(r_n)_{n\in\Z}$ by
$r_0:=1$ and $r_{n+1}:=r_nt_n^{-a}$.
\end{definition}

The parameters $a,h,q,\varepsilon,\delta$ depend only on
$s,\alpha,\beta,\gamma,N$. Constants stated below to be independent
of $T$ depend only on $M,s,\alpha,\beta,\gamma,N$ and on the additional
real parameters specified in the relevant lemma. We fix $T$ after
\cref{lem:uniform}.

\Needspace{12\baselineskip}
\begin{lemma}\label{lem:scales}
The sequence $(r_n)_{n\in\Z}$ is strictly decreasing. We have
$r_n\to0$ as $n\to+\infty$ and $r_n\to\infty$ as $n\to-\infty$.
Consequently the intervals $[r_n,r_{n-1})$ partition $(0,\infty)$.
Moreover,
\[
 |\log r_n|\leq a|n|\log(T+|n|)\leq at_n\log t_n,\qquad
 t_{n-1}\leq2t_n,\quad t_n/t_{n+1}\leq2.
\]
Fix $\lambda_1,\lambda_2>0$ and $v_1,v_2\geq0$.
Let $R_n$ be either $\lambda_1r_nt_n^{-v_1}$ or $\lambda_1r_{n-1}$,
and let $\rho_n$ be either $\lambda_2r_nt_n^{-v_2}$ or
$\lambda_2r_{n-1}$. There are constants $C_1,C_q>0$ with the following
property. They depend only on
$s,\alpha,\beta,\gamma,N,\lambda_1,\lambda_2,v_1,v_2$. We have
\begin{equation}\label{eq:logwindow}
 \Ld(R_n)^\alpha\Ld(\rho_n)^\beta
 \leq C_1(t_n\log t_n)^h\leq C_qt_n^q.
\end{equation}
Furthermore there is a constant $C_*\geq1$. It depends only on
$M,s,\alpha,\beta,\gamma,N$ and has the following property. Every finite $E\subseteq X$
and every $r_n/8$-separated set $Z\subseteq E$ satisfy
\begin{equation}\label{eq:overlap}
 \card{Z\cap B_E(x,r_n/4)}\leq C_*t_n^q\qquad(x\in E).
\end{equation}
\end{lemma}
\begin{proof}
For $m\geq1$, the recurrence relation gives
$r_m=\prod_{j=0}^{m-1}(T+j)^{-a}$ and
$r_{-m}=\prod_{j=1}^{m}(T+j)^a$.
Since the base of each factor is at least $T\geq2$, monotonicity and
both limits follow. Bounding each logarithmic sum by the number of
terms times its largest term gives
$|\log r_n|\leq a|n|\log(T+|n|)$.
The comparisons between neighboring $t_n$ follow from
$\bigl||n+1|-|n|\bigr|\leq1$ and $t_n\geq2$.

Using $|\log(xy)|\leq|\log x|+|\log y|$, we obtain
$\Ld(\lambda r_nt_n^{-v})\leq
1+at_n\log t_n+|\log\lambda|+v\log t_n$.
Using $r_{n-1}=r_nt_{n-1}^a$ and $t_{n-1}\leq2t_n$, we also obtain
$\Ld(\lambda r_{n-1})\leq1+at_n\log t_n+|\log\lambda|+a\log(2t_n)$.
Both expressions are therefore at most $C t_n\log t_n$. The constant
$C$ can be chosen to depend only on $a,\lambda,v$.
Applying these bounds to $R_n$ and $\rho_n$ gives the first inequality
in \eqref{eq:logwindow}. The second follows from
$\sup_{t\geq2}(\log t)^h/t^{q-h}<\infty$.
Finally, applying \cref{lem:packing} with $R=r_n/4$ and $\rho=r_n/8$
and then using \eqref{eq:logwindow} gives \eqref{eq:overlap}.
\end{proof}

\Needspace{16\baselineskip}
\subsection{Separation at one scale}

\Needspace{10\baselineskip}
\begin{definition}[{\cite[Section~2]{MaddisonTarlowMinka2014}}]\label{def:gumbel}
A real-valued random variable $\zeta$ has the \emph{standard Gumbel
distribution} if its distribution function is
$Q(t):=\mathbb P(\zeta\leq t)=\exp(-\mathrm e^{-t})$ for $t\in\R$.
For $A>0$, the distribution obtained by conditioning a standard
Gumbel variable on the event $\{|\zeta|\leq A\}$ is called the
\emph{$A$-truncated standard Gumbel distribution}.
\end{definition}

We use the following elementary properties of the Gumbel distribution.

\Needspace{15\baselineskip}
\begin{lemma}\label{lem:gumbel}
A standard Gumbel variable $\zeta$ satisfies the following properties.
\begin{enumerate}[label=\textup{(\roman*)},leftmargin=2.2em]
\item The variable $\zeta$ has density
$g(t)=\mathrm e^{-t}\exp(-\mathrm e^{-t})$. Moreover
$\sup_{t\in\R}g(t)=\mathrm e^{-1}<1/2$.
\item For every $A\geq0$ we have
$\mathbb P(|\zeta|>A)\leq2\mathrm e^{-A}$.
\item Let $J$ be a nonempty finite set. Let $(\zeta_z)_{z\in J}$
be independent standard Gumbel variables. Let
$(c_z)_{z\in J}\subset\R$. With $b:=\log\sum_{z\in J}\mathrm e^{c_z}$, we have
$\mathbb P(\max_{z\in J}(\zeta_z+c_z)\leq t)=Q(t-b)$ for every
$t\in\R$. Thus the maximum has the distribution of a standard Gumbel variable
shifted by $b$.
\item Suppose that $Y_1,Y_2$ are independent and that each has a
shifted standard Gumbel distribution. Then $Y_1-Y_2$ has a density
bounded above by $1/2$ at every point. Consequently every interval $I\subset\R$ of length $\ell>0$
satisfies $\mathbb P(Y_1-Y_2\in I)\leq\ell/2$.
\item Let $A>0$. Let $\zeta_1,\ldots,\zeta_m$ be independent
standard Gumbel variables. Put
$\mathcal C:=\bigcap_{j=1}^m\{|\zeta_j|\leq A\}$.
Conditionally on $\mathcal C$ the variables $\zeta_1,\ldots,\zeta_m$
remain independent. Each has the $A$-truncated standard Gumbel
distribution.
\end{enumerate}
\end{lemma}
\begin{proof}
\textup{(i)} Differentiating the distribution function $Q$ gives
$Q'(t)=\mathrm e^{-t}\exp(-\mathrm e^{-t})=g(t)$.
With $u=\mathrm e^{-t}>0$ we have $g(t)=u\mathrm e^{-u}$. This function
maximum is $\mathrm e^{-1}$, attained at $u=1$.

\textup{(ii)} The distribution function gives
$\mathbb P(|\zeta|>A)=1-\exp(-\mathrm e^{-A})+\exp(-\mathrm e^A)$.
Applying $1-\mathrm e^{-x}\leq x$ for $x\geq0$ to the first term
on the right gives $1-\exp(-\mathrm e^{-A})\leq\mathrm e^{-A}$.
Also $\mathrm e^A\geq A$ for $A\geq0$. Hence
$\exp(-\mathrm e^A)\leq\mathrm e^{-A}$.
This proves the tail estimate.

\textup{(iii)} By independence we have for every $t\in\R$
$\mathbb P(\max_{z\in J}(\zeta_z+c_z)\leq t)\allowbreak
=\exp(-\mathrm e^{-t}\sum_{z\in J}\mathrm e^{c_z})\allowbreak
=\exp(-\mathrm e^{-(t-b)})=Q(t-b)$. This proves the stated identity.

\textup{(iv)} Let $g_1$ and $g_2$ be the densities of $Y_1$ and $Y_2$.
Both are translates of $g$. Hence
$\sup g_1=\sup g_2=\mathrm e^{-1}<1/2$.
By independence, $Y_1-Y_2$ has density
$h(t)=\int_{\R}g_1(t+y)g_2(y)\,dy$. It satisfies
$h(t)\leq(1/2)\int_{\R}g_2(y)\,dy=1/2$.
Hence $\mathbb P(Y_1-Y_2\in I)=\int_I h(t)\,dt\leq\ell/2$.

\textup{(v)} Let $B_1,\ldots,B_m\subseteq[-A,A]$ be Borel sets.
Independence gives
\[
 \mathbb P\!\left(\bigcap_{j=1}^m\{\zeta_j\in B_j\}\,\middle|\,\mathcal C\right)
 =\frac{\prod_{j=1}^m\mathbb P(\zeta_j\in B_j)}
        {\prod_{j=1}^m\mathbb P(|\zeta_j|\leq A)}
 =\prod_{j=1}^m
   \mathbb P(\zeta_j\in B_j\mid |\zeta_j|\leq A).
\]
Thus the variables remain independent under the conditioning. Each
marginal distribution is the $A$-truncated standard Gumbel distribution
by definition.
\end{proof}

\Needspace{16\baselineskip}
\begin{definition}\label{def:correction}
Let $T\geq2$ and $n\in\Z$. Let $E\subseteq X$ be nonempty and finite.
Put $r:=r_n$, $m_n:=C_*t_n^q$ and $A_n:=\log(8m_n)$. Let
$Z\subseteq E$ be a maximal $r/8$-separated set.
For $U>0$ and real coefficients $\xi_{z,i}\in[-A_n,A_n]$ define
$H=(H_1,\ldots,H_N):E\to\R^N$ by
\begin{equation}\label{eq:h}
 H_i(x):=U\max_{z\in Z}
 \left\{\xi_{z,i}-\frac{16A_n}{r}d(x,z)\right\}
 \qquad(x\in E,\ 1\leq i\leq N).
\end{equation}
We call $H$ a \emph{Gumbel correction map} at scale $r$.
Later, $H$ will be added to a previously constructed map $G$ to separate
pairs of points at scale $r$.
When choosing the coefficients randomly, we take each $\xi_{z,i}$
independently from the $A_n$-truncated standard Gumbel distribution
(\cref{def:gumbel}) and write $\mathbb P$ for the probability measure
determined by these independent choices.
\end{definition}

\Needspace{13\baselineskip}
\begin{lemma}\label{lem:correction}
In the setting of \cref{def:correction}, put
$J_x:=Z\cap B_E(x,r/4)$ for each $x\in E$.
\begin{enumerate}[label=\textup{(\roman*)},leftmargin=2.2em]
\item The set $J_x$ is nonempty. For every choice of coefficients the
maximum in \eqref{eq:h} is attained at a point of $J_x$.
Moreover,
\begin{equation}\label{eq:hbds}
 \sup_{x\in E}\norm{H(x)}_\infty\leq3A_nU,\qquad
 \Lip_\infty(H)\leq16A_nU/r.
\end{equation}
\item Fix a map $G:E\to\R^N$ and points $u,v\in E$ with
$d(u,v)\geq r/2$. Then for every $e>0$
\begin{equation}\label{eq:antic}
 \mathbb P\{\norm{(G+H)(u)-(G+H)(v)}_\infty<4e\}
 \leq(8e/U)^N.
\end{equation}
\end{enumerate}
\end{lemma}
\begin{proof}
\textup{(i)} By \cref{lem:nets}, there exists $z_0\in Z$ with
$d(x,z_0)<r/8$. At this point,
$\xi_{z_0,i}-(16A_n/r)d(x,z_0)>-3A_n$.
If $z\in Z\setminus J_x$ then $d(x,z)\geq r/4$. Hence
$\xi_{z,i}-(16A_n/r)d(x,z)\leq-3A_n$.
Thus the maximum is attained at a point of $J_x$. Moreover
$-3A_n<H_i(x)/U\leq A_n$.
Moreover, the inequality
$|\max_z a_z-\max_z b_z|\leq\max_z|a_z-b_z|$ and the triangle
inequality for the distance function give
$|H_i(x)-H_i(y)|\leq(16A_nU/r)d(x,y)$. This proves the stated bounds.

\textup{(ii)} Fix a coordinate $i$. Consider independent standard Gumbel
variables $\zeta_z$ without conditioning. Denote their probability
measure by $\mathbb P_0$.
The inequality $d(u,v)\geq r/2$ implies $J_u\cap J_v=\varnothing$.
Indeed, a point $z$ in the intersection would give
$d(u,v)\leq d(u,z)+d(z,v)<r/2$.
Take independent $\zeta_z$ indexed by $J_u\cup J_v$. For $x\in\{u,v\}$ put
$M_x:=\max_{z\in J_x}\{\zeta_z-(16A_n/r)d(x,z)\}$. Let
$\mathcal C:=\bigcap_{z\in J_u\cup J_v}\{|\zeta_z|\leq A_n\}$.
The overlap bound gives $\card{J_u\cup J_v}\leq2m_n$. The union
bound and the Gumbel tail estimate then give
$\mathbb P_0(\mathcal C)\geq1-4m_n\mathrm e^{-A_n}=1/2$.
By \cref{lem:gumbel}\textup{(iii)} the variables $M_u$ and $M_v$ are
independent shifted standard Gumbel variables. Part~\textup{(iv)} of the same lemma shows
that $G_i(u)-G_i(v)+U(M_u-M_v)$ has density bounded above by $1/(2U)$.
Let $B_i$ be the event that this variable lies in $(-4e,4e)$. This
interval has length $8e$. Then
$\mathbb P_0(B_i\mid\mathcal C)\leq
\mathbb P_0(B_i)/\mathbb P_0(\mathcal C)\leq8e/U$.
By \cref{lem:gumbel}\textup{(v)} the coefficients remain independent
conditionally on $\mathcal C$. Each has the $A_n$-truncated standard
Gumbel distribution. Their joint distribution therefore agrees with
that of $(\xi_{z,i})_{z\in J_u\cup J_v}$ in \cref{def:correction}.
By \textup{(i)} the values $H_i(u)$ and $H_i(v)$ are determined by this family of
coefficients alone. Thus the preceding estimate also bounds the
probability that $|(G_i+H_i)(u)-(G_i+H_i)(v)|<4e$.
The coefficient families used for coordinates $i=1,\ldots,N$ are
independent. Multiplying these bounds gives \eqref{eq:antic}.
The estimate is independent of $G$. It is also independent of $\card E$
and $\card Z$.
\end{proof}

\Needspace{7\baselineskip}
\begin{definition}[{\cite[Section~1]{MoserTardos2010}}]\label{def:dependency}
Let $(A_i)_{i\in I}$ be a finite family of events in a probability space.
A simple undirected graph on vertex set $I$ is a \emph{dependency graph}
for this family if the following holds. For each $i\in I$ the event
$A_i$ is independent of the $\sigma$-algebra generated by the events
$A_j$ with $j\notin\{i\}\cup\Gamma(i)$. Here $\Gamma(i)$ is the set of
vertices adjacent to $i$.
\end{definition}

We use the following form of the Lov\'asz local lemma.

\begin{lemma}\label{lem:lll}
Suppose that a finite nonempty family of events has a dependency graph
of maximum degree $D$. Suppose also that each event has probability at
most $p_0$. If $4p_0(D+1)\leq1$ then the probability that none of the
events occurs is positive.
\end{lemma}
\begin{proof}
Apply the general local lemma \cite[Theorem~1.1]{MoserTardos2010},
using the same parameter $x:=1/[2(D+1)]$ at each vertex.
Bernoulli's inequality gives $(1-x)^D\geq1-Dx\geq1/2$. Hence
$x(1-x)^D\geq1/[4(D+1)]\geq p_0$.
This verifies the hypotheses of that theorem, including when $D=0$.
\end{proof}

\Needspace{13\baselineskip}
\begin{lemma}\label{lem:separation}
Fix $\lambda>0$ and $\nu\geq0$. There is a constant $C_2>0$ with the following property. The constant
depends only on $M,s,\alpha,\beta,\gamma,N,\lambda,\nu$. Let
$T,n,E,Z$ be as in \cref{def:correction}. Suppose that
$\eta:=\lambda r_nt_n^{-\nu}\leq r_n/4$.
Let $S\subseteq E$ be $\eta$-separated. If $U,e>0$ satisfy
\[
 4(8e/U)^N\left[1+C_2t_n^{2q}
       \left(\frac{r_nr_{n-1}}{\eta^2}\right)^s\right]\leq1,
\]
then the following holds for every fixed map $G:E\to\R^N$. The coefficients in
\eqref{eq:h} can be chosen so that, in addition to \eqref{eq:hbds},
\[
 \norm{(G+H)(u)-(G+H)(v)}_\infty\geq4e
 \quad\left(u,v\in S,\quad r_n/2\leq d(u,v)\leq2r_{n-1}\right).
\]
\end{lemma}
\begin{proof}
Write $r:=r_n$ and $R:=r_{n-1}$. Put
$\mathcal I:=\{\{u,v\}\subseteq S:r/2\leq d(u,v)\leq2R\}$.
If $\mathcal I=\varnothing$, any choice of coefficients suffices.
Assume henceforth that $\mathcal I\ne\varnothing$.
Choose the coefficients as in \cref{def:correction}. For each
$\{u,v\}\in\mathcal I$ let $A_{\{u,v\}}$ be the event
$\{\norm{(G+H)(u)-(G+H)(v)}_\infty<4e\}$.
By \cref{lem:correction},
$\mathbb P(A_{\{u,v\}})\leq(8e/U)^N$. This event depends only on
the coefficients at centers in $J_u\cup J_v$.
Join distinct vertices $\{u,v\}$ and $\{u',v'\}$ in $\mathcal I$ by an edge
if and only if $(J_u\cup J_v)\cap(J_{u'}\cup J_{v'})\ne\varnothing$.
The coefficient family used by each event is independent of the
coefficient family used by all its nonneighbors together. Hence this
graph satisfies \cref{def:dependency}.

Fix a vertex $\{u,v\}$. For an adjacent vertex $\{u',v'\}$,
there are $w\in\{u,v\}$ and $w'\in\{u',v'\}$ such that
$d(w,w')<r/2$. If $w''$ is the other point of $\{u',v'\}$ then
$d(w,w'')\leq d(w,w')+d(w',w'')<r/2+2R<3R$.
Consequently, the number of adjacent vertices is at most
$\sum_{w\in\{u,v\}}\card{S\cap B_E(w,r/2)}\card{S\cap B_E(w,3R)}$.
Each of these sets is $\eta$-separated. Hence
\cref{lem:packing,lem:scales} give
$\card{S\cap B_E(w,r/2)}\leq C_rt_n^q(r/\eta)^s$ and
$\card{S\cap B_E(w,3R)}\leq C_Rt_n^q(R/\eta)^s$.
Here $C_r,C_R$ depend only on
$M,s,\alpha,\beta,\gamma,N,\lambda,\nu$.
Set $C_2:=2C_rC_R$. Then the maximum degree $D_n$ is at most
$C_2t_n^{2q}(r_nr_{n-1}/\eta^2)^s$.
The hypothesis allows us to apply \cref{lem:lll}. Hence there is a choice
of coefficients for which none of the events $A_{\{u,v\}}$ occurs.
The bounds in \cref{lem:correction}\textup{(i)} hold for every choice
of coefficients.
\end{proof}

\Needspace{14\baselineskip}
\subsection{Uniform estimates}
\begin{definition}\label{def:amplitudes}
With $C_*$ as in \cref{lem:scales} put
$K:=\log(8C_*)+q/\varepsilon$ and
$L:=1+16K\sum_{n\in\Z}(2+|n|)^{-(1+\varepsilon)}$. The series
converges because $\varepsilon>0$. For $T\geq2$ and $n\in\Z$ define
$U_n:=r_nt_n^{-(1+2\varepsilon)}$ and $e_n:=r_nt_n^{-(1+a)}$. Also set
$\eta_n:=e_n/(16L)$.
From now on $C_2$ denotes the constant in \cref{lem:separation} with
$\lambda=(16L)^{-1}$ and $\nu=1+a$. The constants $C_*,K,L,C_2$ depend only on
$M,s,\alpha,\beta,\gamma,N$. They are independent of $T$.
\end{definition}

\Needspace{15\baselineskip}
\begin{lemma}\label{lem:uniform}
The following statements hold.
\begin{enumerate}[label=\textup{(\roman*)},leftmargin=2.2em]
\item Every Gumbel correction map $H_n$ defined by \eqref{eq:h} with
$U=U_n$ satisfies
\[
 \sup_{x\in E}\norm{H_n(x)}_\infty
 \leq3Kr_nt_n^{-(1+\varepsilon)},\qquad
 \Lip_\infty(H_n)\leq16Kt_n^{-(1+\varepsilon)}.
\]
For every finite $I\subseteq\Z$, the sum $\sum_{n\in I}H_n$ of such
maps on the same set $E$ is $L$-Lipschitz.
Also, $\eta_n\leq r_n/16$.
\item There exists $T_*\geq2$ with the following property. It depends
only on $M,s,\alpha,\beta,\gamma,N$. For every $T\geq T_*$ and
$n\in\Z$
\begin{equation}\label{eq:LLLparameter}
 4\left(\frac{8e_n}{U_n}\right)^N
 \left[1+C_2t_n^{2q}
       \left(\frac{r_nr_{n-1}}{\eta_n^2}\right)^s\right]\leq1
\end{equation}
and
\[
 6K\sum_{j>n}r_jt_j^{-(1+\varepsilon)}\leq e_n.
\]
\end{enumerate}
\end{lemma}
\begin{proof}
\textup{(i)} For $t\geq1$, we have
$\log t=\int_1^t u^{-1}\,du\leq\int_1^t u^{\varepsilon-1}\,du
=(t^\varepsilon-1)/\varepsilon\leq t^\varepsilon/\varepsilon$.
Hence $A_n=\log(8C_*)+q\log t_n\leq Kt_n^\varepsilon$.
Together with the bounds in \cref{lem:correction}\textup{(i)}, this gives
$\sup_E\norm{H_n}_\infty\leq3Kt_n^\varepsilon U_n
=3Kr_nt_n^{-(1+\varepsilon)}$ and
$\Lip_\infty(H_n)\leq16Kt_n^\varepsilon U_n/r_n
=16Kt_n^{-(1+\varepsilon)}$.
Since $t_n\geq2+|n|$, the definition of $L$ gives
$\Lip_\infty(\sum_{n\in I}H_n)<L$.
The inequality $\eta_n\leq r_n/16$ follows directly from $L\geq1$
and the definition of $\eta_n$.

\textup{(ii)} Since $r_{n-1}/r_n=t_{n-1}^a\leq2^at_n^a$, we have
$r_nr_{n-1}/\eta_n^2\leq(16L)^2 2^at_n^{2+3a}$.
Also, $e_n/U_n=t_n^{-(a-2\varepsilon)}$.
Put $b:=2q+s(2+3a)$. Since $(a-2\varepsilon)N=b+\delta$,
the left-hand side of \eqref{eq:LLLparameter} is at most
\[
 4\cdot8^N t_n^{-(b+\delta)}
 \left[1+C_2(16L)^{2s}2^{as}t_n^b\right]
 \leq C_3t_n^{-\delta}\leq C_3T^{-\delta},
\]
where $C_3:=4\cdot8^N[1+C_2(16L)^{2s}2^{as}]$.

For the tail estimate, put $V_j:=r_jt_j^{-(1+\varepsilon)}$.
For all $j,n\in\Z$,
\[
 \frac{V_{j+1}}{V_j}
 =t_j^{-a}\left(\frac{t_j}{t_{j+1}}\right)^{1+\varepsilon}
 \leq2^{1+\varepsilon}T^{-a},\qquad
 \frac{V_{n+1}}{e_n}
 =t_n^{-\varepsilon}\left(\frac{t_n}{t_{n+1}}\right)^{1+\varepsilon}
 \leq2^{1+\varepsilon}T^{-\varepsilon}.
\]
Choose $T_*\geq2$ so large that
$C_3T_*^{-\delta}\leq1$, $2^{1+\varepsilon}T_*^{-a}\leq1/2$, and
$12K\,2^{1+\varepsilon}T_*^{-\varepsilon}\leq1$.
For $T\geq T_*$ we have $V_{j+1}\leq V_j/2$. Hence
$6K\sum_{j>n}V_j\leq12KV_{n+1}\leq e_n$.
The argument also applies to negative indices.
The constants are chosen in the order $C_*$ then $K,L$ then $C_2,C_3$
and finally $T_*$. None depends on the finite subspace or its cardinality.
\end{proof}

\Needspace{11\baselineskip}
\subsection{Finite subspaces}
\begin{lemma}\label{lem:scalegauge}
Fix $T\geq2$ and put $B:=T+1+(a\log T)^{-1}$.
If $r_n\leq t<r_{n-1}$ then
\[
 t_n\leq B\Ld(t),\qquad
 e_n\geq2^{-a}B^{-\gamma}\frac{t}{\Ld(t)^\gamma}.
\]
\end{lemma}
\begin{proof}
If $n\geq1$ then $t<r_{n-1}\leq1$. Hence
$|\log t|\geq a\sum_{j=0}^{n-2}\log(T+j)\geq a(n-1)\log T$.
The sum is empty when $n=1$.
If $n\leq-1$ then $t\geq r_n>1$. Hence
$|\log t|\geq a\sum_{j=n}^{-1}\log(T+|j|)\geq a|n|\log T$.
This also covers the case $n=0$ and gives
$|n|\leq1+|\log t|/(a\log T)$. Hence
$t_n=T+|n|\leq B\Ld(t)$.
Moreover, $t<r_{n-1}\leq2^ar_nt_n^a$ and $1+2a=\gamma$ give
$e_n=r_nt_n^{-(1+a)}\geq2^{-a}t\,t_n^{-\gamma}
\geq2^{-a}B^{-\gamma}t\,\Ld(t)^{-\gamma}$.
\end{proof}

\Needspace{9\baselineskip}
\begin{proposition}\label{prop:finite}
There are constants $c_0,L>0$ with the following property. They depend
only on $M,s,\alpha,\beta,\gamma,N$. For every finite subspace
$E\subseteq X$ there is a map $F_E:E\to\R^N$ satisfying
\begin{equation}\label{eq:finite}
 c_0\frac{d(x,y)}{\Ld(d(x,y))^\gamma}
 \leq\norm{F_E(x)-F_E(y)}_\infty
 \leq Ld(x,y)\qquad(x,y\in E,\ x\ne y).
\end{equation}
These constants are independent of $\card E$ and $\diam E$. They are
also independent of the smallest nonzero distance in $E$.
\end{proposition}
\begin{proof}
Fix the constants $L$ and $T=T_*$ chosen above.
Take the zero map when $E$ is empty or a singleton. Otherwise
\[
 I_E:=\{n\in\Z:\text{there exist }x,y\in E\text{ with }x\ne y,\
                         r_n\leq d(x,y)<r_{n-1}\}
\]
is a nonempty finite set. Put $n_-:=\min I_E$ and $n_+:=\max I_E$. Choose maps $H_n$ successively for every index
$n_-,n_-+1,\ldots,n_+$.

Once $H_j$ has been chosen for $n_- \leq j<n$, put
$G_n:=\sum_{j=n_-}^{n-1}H_j$. When $n=n_-$ this is the empty sum and $G_{n_-}=0$. Choose a maximal $r_n/8$-separated set $Z_n$ and a
maximal $\eta_n$-separated set $S_n$ in $E$.
By \cref{lem:uniform}, we may apply \cref{lem:separation} to the fixed
map $G_n$ with $U=U_n$ and $e=e_n$. We also take $\eta=\eta_n$.
We can therefore choose $H_n$ satisfying the bounds in
\cref{lem:uniform}\textup{(i)} and
$\norm{(G_n+H_n)(u)-(G_n+H_n)(v)}_\infty\geq4e_n$ whenever
$u,v\in S_n$ and $r_n/2\leq d(u,v)\leq2r_{n-1}$.
Since the probability estimate in \cref{lem:correction}\textup{(ii)}
holds for every fixed $G_n$, the maps $H_n$ may
be chosen successively while retaining all previous choices.
Set $F_E:=\sum_{n=n_-}^{n_+}H_n$.
By \cref{lem:uniform}, $\Lip_\infty(F_E)\leq L$.

Fix distinct $x,y\in E$. Let $n$ be the index such that
$r_n\leq d(x,y)<r_{n-1}$. Since $S_n$ is an $\eta_n$-net there are $u,v\in S_n$ such that
$d(x,u)<\eta_n$ and $d(y,v)<\eta_n$. The bound $\eta_n\leq r_n/16$
gives
$d(u,v)>r_n-2\eta_n\geq7r_n/8>r_n/2$ and
$d(u,v)<r_{n-1}+2\eta_n<2r_{n-1}$.
Put $Q_n:=G_n+H_n$. This partial sum is also $L$-Lipschitz. Hence
$\norm{Q_n(x)-Q_n(y)}_\infty\geq4e_n-e_n/8$.
The first estimate in \cref{lem:uniform} and the tail bound in
part~\textup{(ii)} give
$\norm{\sum_{j=n+1}^{n_+}(H_j(x)-H_j(y))}_\infty\leq e_n$.
The triangle inequality yields
$\norm{F_E(x)-F_E(y)}_\infty\geq(4-1/8-1)e_n=(23/8)e_n\geq2e_n$.
Apply \cref{lem:scalegauge} and take $c_0:=2^{1-a}B^{-\gamma}$. This
gives \eqref{eq:finite}. The constants have the asserted dependence
because $T=T_*$ and $B$ depend only on $M,s,\alpha,\beta,\gamma,N$.
\end{proof}

\subsection{General metric spaces}

\begin{proof}[Proof of \cref{thm:quant}]
Fix $\gamma,N$ satisfying \eqref{eq:quantcond}.
By \cref{prop:finite} there are constants $c_0,L$ that work for all
finite subspaces of $X$. If $X$ is finite, apply that proposition and
\cref{lem:euclidean}\textup{(i)}.
Assume henceforth that $X$ is infinite.

By \cref{prop:basic}, there is a countable dense set
$D=\{x_0=o,x_1,\ldots\}$ of distinct points.
Put $E_m:=\{x_0,\ldots,x_m\}$. Translate the map
$F_m:E_m\to\R^N$ from \cref{prop:finite} so that $F_m(o)=0$.
For fixed $j$ and $m\geq j$ the inequality
$\norm{F_m(x_j)}_\infty\leq Ld(x_j,o)$ places $F_m(x_j)$ in the
compact set $[-Ld(x_j,o),Ld(x_j,o)]^N$.
Successively taking convergent subsequences for each $j$ and then a
diagonal subsequence $(F_{m_k})$, we obtain convergence at every point
of $D$. Define $f(x_j):=\lim_{k\to\infty}F_{m_k}(x_j)$.
Any fixed $x_i,x_j$ belong to $E_{m_k}$ for all sufficiently large $k$.
Continuity of the norm therefore shows that both inequalities in
\eqref{eq:finite} also hold for $f$.

For each $x\in X$ choose a sequence $(y_k)\subset D$ converging
to $x$. Since
$\norm{f(y_k)-f(y_\ell)}_\infty\leq
L(d(y_k,x)+d(y_\ell,x))\to0$, the sequence $f(y_k)$ converges by
completeness of $\R^N$.
If $(z_k)\subset D$ is another sequence converging to $x$ then
$\norm{f(y_k)-f(z_k)}_\infty\leq L(d(y_k,x)+d(z_k,x))\to0$.
Thus $F(x):=\lim_{k\to\infty}f(y_k)$ is independent of the
approximating sequence.

For distinct $x,y\in X$ choose sequences in $D$ with $y_k\to x$
and $z_k\to y$. The convergence $d(y_k,z_k)\to d(x,y)>0$ implies
$y_k\ne z_k$ for all sufficiently large $k$.
Use continuity of the distance and the maximum norm. Also use continuity
of $t\mapsto t/\Ld(t)^\gamma$ for $t>0$. Passing to the limit in both
inequalities for $f(y_k),f(z_k)$. This gives \eqref{eq:finite} for $F$.
Finally \cref{lem:euclidean}\textup{(i)} yields \eqref{eq:abl} with
$c=c_0$ and $C=\sqrt N L$.
These constants depend only on $M,s,\alpha,\beta,\gamma,N$.
\end{proof}

\begin{proof}[Proof of \cref{cor:doubling}]
By \cref{prop:basic}, a $K$-doubling space satisfies \eqref{eq:ah}
with $M=K$ and $s=\log_2K$. Here $\alpha=\beta=0$.
Substituting $h=0$ into \eqref{eq:quantcond} gives the asserted
dimension condition.
\end{proof}

\begin{proof}[Proof of \cref{thm:char}]
For \textup{(i)} $\Rightarrow$ \textup{(ii)} choose
$M,s,\alpha,\beta$ for which $X$ satisfies \eqref{eq:ah}. Then take
$\gamma>1$ and a positive integer $N$ satisfying \eqref{eq:quantcond}.
Apply \cref{thm:quant}.

The implication \textup{(ii)} $\Rightarrow$ \textup{(i)} is known. We
include a proof here for completeness.
Suppose that $F:X\to\R^N$ satisfies \eqref{eq:abl}.
Fix $x\in X$ and $0<\rho<R$. Let $S\subseteq B_X(x,R)$ be a
finite $\rho$-separated set.
By \cref{lem:gauge} the set $F(S)$ is $\sigma$-separated with
$\sigma:=c\kappa_\gamma\rho/\Ld(\rho)^\gamma$. Also
$F(S)\subseteq B_{\R^N}(F(x),CR)$. The map $F$ is injective.
The Euclidean packing estimate gives
\[
 \card S\leq\left(1+\frac{2CR}{\sigma}\right)^N
 \leq C_4\left(\frac R\rho\right)^N\Ld(\rho)^{\gamma N},
 \qquad C_4:=\left(1+\frac{2C}{c\kappa_\gamma}\right)^N.
\]
For the last inequality we used $1+bA\leq(1+b)A$ with
$A:=(R/\rho)\Ld(\rho)^\gamma\geq1$ and
$b:=2C/(c\kappa_\gamma)>0$.
The cardinalities of all finite $\rho$-separated sets have a common
upper bound. Hence \cref{lem:nets} gives a finite maximal
$\rho$-separated subset of $B_X(x,R)$. This set is a $\rho$-net.
Using its points as centers of covering balls gives \eqref{eq:ah}
with $(M,s,\alpha,\beta)=(C_4,N,0,\gamma N)$.
\end{proof}

\Needspace{7\baselineskip}
\begin{corollary}\label{cor:diff}
Every almost homogeneous metric space $X$ admits an almost bi-Lipschitz
embedding $F:X\to\R^N$ whose image has a $5^N$-doubling difference set
$F(X)-F(X):=\{F(x)-F(y):x,y\in X\}$.
\end{corollary}
\begin{proof}
Take $F$ from \cref{thm:quant} and apply
\cref{lem:euclidean}\textup{(iii)} to $F(X)-F(X)\subseteq\R^N$.
\end{proof}

This corollary answers the image-difference-set questions in
\cite[Section~9, Question~(3)]{OlsonRobinson2010} and
\cite[Question~4.2]{Margaris2020v1} affirmatively with a doubling
image difference set. For the latter question the embedding is allowed
to be nonlinear. The question number refers to the first arXiv version
\cite{Margaris2020v1}.

\section*{Acknowledgments}
The author used the GPT-6.0 model to prepare this manuscript. The author reviewed and revised the mathematical content and takes full responsibility for the final manuscript.

\bigskip
\noindent\textsc{Mathematical Institute, Tohoku University}\\
Sendai 980-8578, Japan\\
\textit{Email address:} \texttt{fukano.ryo.t5@dc.tohoku.ac.jp}

\end{document}